\documentclass{amsart}
\usepackage{amsmath,amsfonts,amsthm}
\usepackage{amssymb,latexsym}
\usepackage{graphics}
\usepackage[colorlinks]{hyperref}

\theoremstyle{plain}
\theoremstyle{definition}
\newtheorem{theorem}{Theorem}[section]
\newtheorem{lemma}[theorem]{Lemma}

\newtheorem{corollary}[theorem]{Corollary}

\newtheorem{definition}[theorem]{Definition}
\newtheorem{example}[theorem]{Example}

\newtheorem{problem}[theorem]{Problem}

\newtheorem{note}[theorem]{Note}

\newtheorem{convention}[theorem]{Convention}
\newtheorem{remark}[theorem]{Remark}
\theoremstyle{remark}

\numberwithin{equation}{section}
\newcommand{\SP}{\: \: \: \: \:}

\title{On pseudo--co--decomposition of a transformation group}
\author[S. Arzanesh, F. Ayatollah Zadeh Shirazi, R. Rezavand]{Safoura Arzanesh, Fatemah Ayatollah Zadeh Shirazi, Reza Rezavand}
\begin{document}
\begin{abstract}
In the following text we introduce the concept of pseudo--co--decomposition of a transformation group,
also we show the collection of all transformation groups pseudo--co--decomposable to distal ones is a proper intermediate class
of the class of all transformation groups and the class of all transformation groups co--decomposable to distal ones.
\end{abstract}
\maketitle
\noindent {\small {\bf 2020 Mathematics Subject Classification:}  54H15 \\
{\bf Keywords:}} Co--decomposition, Distal, Pseudo--co--decomposition, Transformation (semi)group.
\section{Introduction}
\noindent  ``Classifying'' the objects of a category, may be occurred in different ways. One may introduce (finite or infinite) chains
of objects, e.g. in TOP (category of topological spaces) this is a chain (under inclusion relation):
``The class of metrizable spaces $\subset$ The class of T$_2$ spaces $\subset$ The class of  T$_1$ spaces
$\subset$ The class of  T$_0$ spaces $\subset$ The class of  topological spaces''.
\\
Sometimes try to refine an existing classification leads to new
concepts. However one can mention several texts envolving the
matter in topological dynamics as well  as other areas of
mathematics~\cite{golestani, hj}. In this point of view for
dynamical property $\mathsf P$ we introduce the class of
transformation (semi)groups pseudo--co--demoposable to
tansformation semigroups carrying property $\mathsf P$ (with
emphasis on case $\mathsf P$ is distality/ non--minimality).
\\
In an other point of view, some authors try to understand a transformation semigroup by ``dividing'' its phase space~(\cite[Proposition 2.6]{ellis} and~\cite{chu}). In this text following~\cite{decom} we approach the matter by
``dividing'' phase (semi)group.
\\
In the following ``$\subset$'' means strict inclusion.
\subsection*{Co--decomposition vs pseudo--co--decomposition} For nonempty collection
$\{(X,S_\alpha):\alpha\in\Gamma\}$ of transformation semigroups
the concept of multi--transformation semigroup
$(X,(S_\alpha:\alpha\in\Gamma))$ has been introduced as a
generalization of bitransformation group for the first time
in~\cite{decom}, which leads to the definition of
\linebreak
co--decomposition of a transformation semigroup. In this way for
a transformation semigroup $(X,S)$ and dynamical property
$\mathsf{P}$ one may ask ``Is there any 
co--decomposition of
transformation semigroup $(X,S)$ like
$(X,(S_\alpha:\alpha\in\Gamma))$ such that for each
$\alpha\in\Gamma$, $(X,S_\alpha)$ has property $\mathsf{P}$?'' in
other words we want to know whether $(X,S)$ is co--decomposable
to $\mathsf{P}$ transformation semigroups. Co--decomposability of
a transformation semigroup to those carrying a special property
brought an other way to classify transformation semigroups, e.g.
we have the following diagram for $\mathsf{P}\in\{$distal,
equicontinuous$\}$~\cite{decom}: {\small
\\
\begin{tabular}{l}
$\:$ \\
The class of all $\mathsf{P}$ transformation semigroups
\\
\begin{tabular}{rl}
& $\subset$
The class of all transformation semigroups
co--decomposable to $\mathsf{P}$ ones \\
& $\subset$ The class of all transformation semigroups. \\
\end{tabular}
\\ $\:$ \\
\end{tabular}
\\ 
In this text we introduce the concept of pseudo--multi--transformation semigroup and pseudo--co--decomposability
to $\mathsf{P}$ transformation semigroups and via examples show
that we have the following strict inclusions:
{\small
\\
\begin{tabular}{l}
$\:$ \\
The class of all distal transformation semigroups  \\
\begin{tabular}{rl}
& $\subset$
The class of all transformation semigroups
co--decomposable to distal ones \\
& $\subset$
The class of all transformation semigroups
pseudo--co--decomposable to distal ones \\
& $\subset$  The class of all transformation semigroups. \\
\end{tabular} \\ $\:$ \\ \end{tabular}}
\noindent
In this paper our aim is to bring a more precise tool
(pseudo--co--decomposition) to classify transformation semigroups.
\section{Preliminaries}
\noindent By a transformation semigroup (group) $(X,S,\pi)$ or
simply $(X,S)$ we mean a compact Hausdorff topological space $X$,
a discrete topological semigroup (group) $S$ with identity $e$
and continuous map $\mathop{\pi:X\times S\to
X}\limits_{\SP\SP\SP(x,s)\mapsto xs}$ such that for all $x\in X$
and $s,t\in S$ we have $xe=x$ and $x(st)=(xs)t$. In
transformation semigroup $(X,S)$ we say $S$ acts effectively on
$X$, if for each distinct $s,t\in S$ there exists $x\in X$ with
$xs\neq xt$~\cite{ellis}.
\\
For nonempty collection $\{(X,S_\alpha):\alpha\in\Gamma\}$
of transformation semigroups (groups) we
say $(X,(S_\alpha:\alpha\in\Gamma))$ is a multi--transformation
semigroup (group) if for every distinct $\alpha_1,\ldots,\alpha_n
\in\Gamma$, $x\in X$, $s_1\in S_{\alpha_1},\ldots,s_n\in S_{\alpha_n}$ and permutation
$\mathop{\{1,\ldots,n\}\to
\{1,\ldots,n\}}\limits_{k\mapsto m_k}$ we have~\cite{decom}:
\[(\cdots((xs_1)s_2)\cdots)s_n=
(\cdots((xs_{m_1})s_{m_2})\cdots)s_{m_n}\:.\]
In transformation semigroup (group) $(X,S)$ we say
multi--transformation semigroup (group) $(X,(S_\alpha:\alpha\in\Gamma))$ is a co--decomposition of $(X,S)$ if
$S_\alpha$s are distinct sub--semigroups (sub--groups)
of $S$ and semigroup (group) generated by $\bigcup\{S_\alpha:
\alpha\in\Gamma\}$ is $S$.
\\
For dynamical property $\mathsf{P}$ we say the transformation
semigroup (group) $(X,S)$ is co--decomposable to $\mathsf{P}$
transformation semigroups (groups) if it has a co--decomposition
like $(X,(S_\alpha:\alpha\in\Gamma))$ to transformation
semigroups (groups) such that for all $\alpha\in\Gamma$,
$(X,S_\alpha)$ has property $\mathsf{P}$.
\\
We generalize the concept of multi--transformation semigroup and
co--decomposition of a transformation semigroup simply in the
following way:
\begin{definition}
For nonempty collection $\{(X,S_\alpha):\alpha\in\Gamma\}$
of transformation semigroups (groups) we
say $(X,(S_\alpha:\alpha\in\Gamma))$ is a pseudo--multi--transformation
semigroup (group) if for every $\alpha_1,\ldots,\alpha_n
\in\Gamma$, $x\in X$ and permutation
$\mathop{\{1,\ldots,n\}\to
\{1,\ldots,n\}}\limits_{k\mapsto m_k}$ we have:
\[(\cdots((xS_{\alpha_1})S_{\alpha_2})\cdots)S_{\alpha_n}=
(\cdots((xS_{\alpha_{m_1}})S_{\alpha_{m_2}})\cdots)
S_{\alpha_{m_n}}\:.\]
In transformation semigroup (group) $(X,S)$ we say
pseudo--multi--transformation semigroup (group) $(X,(S_\alpha:\alpha\in\Gamma))$ is a co--decomposition of $(X,S)$ if
$S_\alpha$s are distinct sub--semigroups (sub--groups)
of $S$ and semigroup (group) generated by $\bigcup\{S_\alpha:
\alpha\in\Gamma\}$ is $S$.
\\
For dynamical property $\mathsf{P}$ we say the transformation
semigroup (group) $(X,S)$ is 
\linebreak
pseudo--co--decomposable to $\mathsf{P}$
transformation semigroups (groups) if it has a 
\linebreak
pseudo--co--decomposition
like $(X,(S_\alpha:\alpha\in\Gamma))$ to transformation
semigroups (groups) such that for all $\alpha\in\Gamma$,
$(X,S_\alpha)$ has property $\mathsf{P}$.
\end{definition}
\section{Pseudo--co--decompositions a new way to classify transformation semigroups}
\noindent It's evident that every multi--transformation semigroup
is a pseudo--multi--transformation semigroup, and
every co--decomposition is a pseudo--co--decomposition. In this section we present a
pseudo--multi--transformation semigroup (pseudo--co--decomposition
of a transformation semigroup), which is not a
multi--transformation semigroup (co--decomposition). Then in
subsections we present counterexamples to show considering
\linebreak
pseudo--co--decompositions provides finer classifications on the
class of transformation groups with emphasize on distality and
non--point transitivity.
\begin{convention}
Let $\mathfrak{B}=\{\frac1n:n\geq1\}\cup\{0\}$ with induced
topology of Euclidean real line $\mathbb{R}$. Also
for permutation $\sigma:\mathop{\mathbb{N}\to\mathbb{N}}\limits_{
k\mapsto m_k}$ consider $\mathsf{f}_\sigma:\mathfrak{B}\to
\mathfrak{B}$ with:
\[x\mathsf{f}_\sigma:=\left\{\begin{array}{lc}
& \\ \frac{1}{m_k} & x=\frac1k,k\geq1\:, \\ & \\ 0 & x=0\:. \\ & 
\end{array}\right.\]
Thus $h:\mathfrak{B}\to\mathfrak{B}$ is a homeomorphism
if and only if there exists a permutation $\tau:\mathbb{N}\to
\mathbb{N}$ with $h=\mathsf{f}_\tau$.
\\
For $n\geq1$ let $T_n=\{f:\mathfrak{B}\mathop{\to}\limits^{f}\mathfrak{B}$ is bijective and for all $x\neq 0,1,\frac{1}{2},\ldots,\frac{1}{n}$ we have $xf=x\}$
(thus $T_n=\{\mathsf{f}_\sigma:\forall k\geq n\:k\sigma=k\}$)
and $T=\bigcup\{T_m:m\geq1\}$. Thus $T$ (and $T_n$s) is a group
of homeomorphisms on $\mathfrak{B}$ (under the operation
of composition of maps) which acts in a natural way on
$\mathfrak{B}$.
\end{convention}
\begin{example}\label{salam10}
$(\mathfrak{B},(T_n:n\geq1))$ is a pseudo--co--decomposition of
$(\mathfrak{B},T)$ and it is not a co--decomposition of
$(\mathfrak{B},T)$.
In particular, $(\mathfrak{B},(T_n:n\geq1))$ is a
pseudo--multi--transformation semigroup and it is not a multi--transformation semigroup.
\end{example}
\begin{proof}
Suppose $n_1,\ldots,n_k\geq1$ and $p=\max(n_1,\ldots,n_k)$, then
for every permutation $\mathop{\{1,\ldots,k\}\to\{1,\ldots,k\}}\limits_{j\mapsto m_j}$ we have $T_{n_1}\cdots T_{n_k}=T_p=T_{n_{m_1}}\cdots T_{n_{m_k}}$, hence for all $x\in \mathfrak{B}$ we have
$xT_{n_1}\cdots T_{n_k}=xT_p=xT_{n_{m_1}}\cdots T_{n_{m_k}}$ and
$(\mathfrak{B},(T_n:n\geq1))$ is a pseudo--multi--transformation semigroup.
\\
Now consider permutations $\sigma=(1 \: 2)$ and $\mu=(1 \: 3)$ on
$\mathbb{N}$, then $\mathsf{f}_\sigma\in T_2$ and
$\mathsf{f}_\mu\in T_3$
and
$1\mathsf{f}_\sigma\mathsf{f}_\mu=\frac12\neq\frac13=1
\mathsf{f}_\mu\mathsf{f}_\sigma$,
so $(\mathfrak{B},(T_n:n\geq1))$ is not a multi--transformation semigroup.
\end{proof}
\subsection{Is $(\mathfrak{B},T)$ (pseudo--)co--decomposable to distal transformation semigroups?}
\noindent In this sub--section we show $(\mathfrak{B},T)$
is pseudo--co--decomposable to distal
transformation groups, however it is not co--decomposable
to distal transformation semigroups.
\\
In transformation semigroup $(X,S)$ consider proximal
relation
$P(X,S)=\{(x,y)\in X\times X:$ there exists a net $\{t_\alpha\}_{
\alpha\in\Gamma}$ is $S$ and $z\in X$ such that
$\mathop{\lim}\limits_{\alpha\in\Gamma}xt_\alpha=z=
\mathop{\lim}\limits_{\alpha\in\Gamma}yt_\alpha\}$. We
say the transformation semigroup $(X,S)$ is distal
if $P(X,S)=\Delta_X(=\{(x,x):x\in X\})$~\cite{ellis}.
\begin{lemma}\label{salam50}
In transformation semigroup $(\mathfrak{B},S)$ suppose
$S\subseteq\{\mathsf{f}_\sigma:\sigma$ is a permutation on
$\mathbb{N}\}$. The following statements are equivalent:
\begin{itemize}
\item[1.] $(\mathfrak{B},S)$ is distal,
\item[2.] $(\mathfrak{B},S)$ has a finite pseudo-co-decomposition
    $(\mathfrak{B},(S_i:1\leq i\leq n))$ to distal
    transformation semigroups,
\item[3.] for all $x\in \mathfrak{B}$, $xS$ is finite,
\item[4.] for all $x\in \mathfrak{B}\setminus\{0\}$, $xS$ is finite.
\end{itemize}
\end{lemma}
\begin{proof} First of all note that $0S=\{0\}$ is finite, hence (3) and (4) are equivalent.
\\
``(1)$\Rightarrow$(4)'' If there exists $x\in \mathfrak{B}\setminus\{0\}$
with infinite $xS$, then there exists a sequence $\{s_n\}_{n\geq1}$ in $S$ such that
$\{xs_n\}_{n\geq1}$ is a one--to--one sequence in $xS\subseteq\{\frac1n:n\geq1\}=\mathfrak{B}\setminus\{0\}$
thus $\mathop{\lim}\limits_{n
\to\infty}xs_n=0(=\mathop{\lim}\limits_{n\to\infty}0s_n)$,
and $(0,x)\in P(\mathfrak{B},S)$, hence $(\mathfrak{B},S)$
is not distal.
\\
``(3)$\Rightarrow$(1)'' Suppose for all $x\in \mathfrak{B}$, $xS$
is finite and $(u,v)\in P(\mathfrak{B},S)$, then there exists a
sequence $\{t_n\}_{n\geq1}$ in $S$ such that
$z:=\mathop{\lim}\limits_{n\to\infty}ut_n=
\mathop{\lim}\limits_{n\to\infty}vt_n$. Thus $uS\cap
vS=\overline{uS}\cap\overline{vS}\ni z$. If $u\neq v$, then we
may suppose $u\neq0$ hence $0\notin uS$ and $z\neq 0$ which leads
to openness of $\{z\}$. There exists $N\geq1$ such that for all
$n\geq N$ we have $ut_n=z=vt_n$ which shows $u=v$ and distality
of $(\mathfrak{B},S)$.
\\
``(1)$\Rightarrow$(2)'' It's clear that (1) implies (2).
\\
``(2)$\Rightarrow$(4)''
Suppose $(\mathfrak{B},(S_i:1\leq i\leq n))$ is a
pseudo--co--decomposition of $(\mathfrak{B},S)$ to distal
transformation semigroups, then for each $1\leq i\leq n$, $(X,S_i)$ is distal and $xS_i$ is finite for each $x\in X$
(since (1) and (3) are equivalent). Thus for each $x\in X$, $xS_1\cdots S_n$ is finite, using
definition of pseudo-co-decomposition we have $xS=xS_1\cdots S_n$, hence $xS$ is finite
which shows distality of $(\mathfrak{B},S)$.
\end{proof}
\begin{lemma}\label{salam20}
In transformation semigroup $(X,S)$ suppose $S$ acts effectively on $X$ and $(X,(S_\alpha:\alpha\in\Gamma))$ is a co--decomposition of $(X,S)$, in particular suppose $e\in\bigcap\{S_\alpha:\alpha\in\Gamma\}$, then
for each $\alpha,\beta\in\Gamma$ we have (for $A\subseteq S$ suppose $<A>$ is the sub--semigroup
of $S$ generated by $A$):
\begin{itemize}
\item[1.] $<S_\alpha \cup S_\beta>=S_\alpha S_\beta$,
\item[2.] $S=\bigcup\{S_{i_1}\cdots S_{i_n}:n\geq1$ and
    $i_1,\ldots,i_n\in\Gamma$ are distinct$\}$,
\item[3.] there exists
    a subset $\Lambda$ of $\Gamma$
    with ${\rm card}(\Lambda)\leq\max(\aleph_0,{\rm card}(S))$
    such that $(X,(S_\alpha:\alpha\in\Lambda))$ is a co--decomposition of $(X,S)$.
\end{itemize}
\end{lemma}
\begin{proof} 1, 2) Note that for each distinct $\alpha,\beta\in\Gamma$, $s_\alpha\in S_\alpha$ and $s_\beta\in S_\beta$ we have:
\[\forall x\in X\:\: xs_\alpha s_\beta=xs_\beta s_\alpha\]
thus $s_\alpha s_\beta=s_\beta s_\alpha$ since $(X,S)$ is effective.
\\
3)  Consider $\theta\in\Gamma$. Using (2), for each $t\in S$ we may choose a finite
sequence $(\alpha_1^t,\cdots,\alpha_{n_t}^t)$ of elements
of $\Gamma$ such that
\[t\in S_{\alpha_1^t}\cdots S_{\alpha_{n_t}^t}\:.\]
Let
$(w_k^t)_{k\geq1}:=(\alpha_1^t,\cdots,\alpha_{n_t}^t,\theta,\theta,\cdots)$. Then for each $k\geq1$, $\Lambda_k=\{w_k^t:t\in S\}$ is a subset of $\Gamma$ with ${\rm card}(\Lambda_k)\leq
{\rm card}(S)$. Let $\Lambda=\bigcup\{\Lambda_k:k\geq1\}$,
then ${\rm card}(\Lambda)\leq\max(\aleph_0,{\rm card}(S))$.
Moreover $S$ is sub--semigroup generated by $\{S_i:i\in\Lambda\}$, thus
$(X,(S_\alpha:\alpha\in\Lambda))$ is a co--decomposition of $(X,S)$.
\end{proof}
\begin{theorem}\label{salam60}
The transformation semigroup $(\mathfrak{B},T)$ is not
co--decomposable to distal transformation semigroups.
\end{theorem}
\begin{proof}
First note that for all $x\in \mathfrak{B}\setminus\{0\}$ the set
$xT=\mathfrak{B}\setminus\{0\}$ is infinite, thus
$(\mathfrak{B},T)$ is not distal by Lemma~\ref{salam50}. Suppose
$(\mathfrak{B},(S_\alpha:\alpha\in\Gamma))$ is a
co--decomposition of $(\mathfrak{B},T)$ to distal transformation
semigroups, then by item (3) of Lemma~\ref{salam20} and
countability of $T$, there exists a countable subset $\Lambda$ of
$\Gamma$ such that $(\mathfrak{B},(S_\alpha:\alpha\in\Lambda))$
is a co--decomposition of $(\mathfrak{B},T)$ (to distal
transformation semigroups). Since $(\mathfrak{B},T)$ is not
distal, by Lemma~\ref{salam50}, $\Lambda$ is not finite, hence
$\Lambda$ is infinite countable.
\\
So we may suppose $\Lambda=\mathbb{N}$ and
$(\mathfrak{B},(S_n:n\geq1))$ is a co--decomposition of $(X,T)$
to distal transformation semigroups. Since $T_3$ is a finite
subset of $T(=\bigcup\{S_1\cdots S_n:n\geq1\})$, there exists
$p\geq1$ such that $T_3\subseteq S_1\cdots S_p$. Suppose
$\{K_n:n\geq1\}=\{S_n:n>p\}\cup\{S_1\cdots S_p\}$ with
$K_1=S_1\cdots S_p$ and distinct $K_n$s. Then
$(\mathfrak{B},(K_n:n\geq1))$ is a co--decomposition of $(X,T)$
to distal transformation semigroups with $T_3\subseteq K_1$.
Let's continue the proof via the following claims:
\\
{\bf Claim 1.} $1K_n=\{1\}$ for all $n\geq2$.
\\
{\it Proof of Claim 1.} Consider $n\geq2$, note that $1K_n\subseteq 1T=\mathfrak{B}\setminus\{0\}=\{\frac1t:t\geq1\}$.
For $t>1$  if $\frac1t\in 1K_n$ then
there exists permutation $\mu:\mathbb{N}\to\mathbb{N}$ with
$1\mu=t$ and $\mathsf{f}_\mu\in K_n$, so for
permutation $\theta=(1\: \: q)$ with $q\in\{2,3\}\setminus\{t\}$
we have $\mathsf{f}_\theta\in T_3\subseteq K_1$, hence
$\frac1q\mathsf{f}_\theta\mathsf{f}_\mu=
\frac1q\mathsf{f}_\mu\mathsf{f}_\theta$ thus
$\frac{1}{1\mu}=\frac{1}{q\mu\theta}$ and
$t=q\mu\theta$ hence $q\mu=t\theta^{-1}=t=1\mu$ which
leads to contradiction $q=1$. Hence $\frac1t\notin 1K_n$
and $1K_n=\{1\}$.
\\
{\bf Claim 2.} For $s,n\geq2$ with  $ \frac1s   \in 1K_1$ we have $\frac1s K_n=\{\frac1s\}$.
\\
{\it Proof of Claim 2.} For $n,s\geq2$ and $t\geq1$ suppose
$\frac1s \in 1K_1$ and $\frac1t\in\frac1s K_n$. There exist permutations $\mu,\lambda$ on
$\mathbb N$ with $1\mu=s$, $s\lambda=t$, $\mathsf{f}_\mu\in K_1$
and $\mathsf{f}_\lambda\in K_n$. By Claim 1, $\frac1{1\lambda}=1\mathsf{f}_\lambda\in1K_n=\{1\}$ and  $1\lambda=1$. So $\frac{1}{1\mu\lambda}=
1\mathsf{f}_\mu\mathsf{f}_\lambda=1
\mathsf{f}_\lambda\mathsf{f}_\mu
=\frac{1}{1\lambda\mu}$ and
$t=s\lambda=1\mu\lambda=1\lambda\mu= 1\mu=s$. Hence
$\frac1s K_n=\{\frac1s\}$.
\\
Using Lemma~\ref{salam50} (since $(\mathfrak{B},K_1)$ is distal),
$1K_1$ is finite and there exists $q>1$ with
$\frac1q\in\mathfrak{B}\setminus 1K_1$. Consider
permutation $\psi=(1\: q)$ on $\mathbb{N}$ then
$\mathsf{f}_\psi\in T(=\bigcup\{K_1\cdots K_n:n\geq1\})$
thus there exists $m\geq1$, $h_1\in K_1,\ldots,h_m\in K_m$
with $\mathsf{f}_\psi=h_1\cdots h_m$.
For all $i\in\{2,\ldots,m\}$ using Claim 1, $1h_i\in 1K_i=\{1\}$, thus
$1h_1h_i=1h_ih_1=1h_1$, and $1h_1=1h_1\cdots h_m=1\mathsf{f}_\psi=\frac{
1}{1\psi}=\frac1q$ which leads to contradiction $\frac1q\in1K_1$.
\\
Therefore $(\mathfrak{B},T)$ is not co--decomposable to distal transformation semigroups.
\end{proof}
\begin{example}\label{salam70}
Using Theorem~\ref{salam60}, $(\mathfrak{B},T)$ is not
co--decomposable to distal transformation semigroups, however for
all $n\geq1$, $T_n$ is a finite group and $(\mathfrak{B}, T_n)$
is a distal transformation group, so $(\mathfrak{B},
(T_n:n\geq1))$ is a pseudo--co--decomposition of $(
\mathfrak{B},T)$ to distal transformation groups (see
Example~\ref{salam10}).
\end{example}
\begin{example}\label{salam80}
Let $G=\{\mathsf{f}_\sigma:\sigma$ is a permutation on $\mathbb{N}\}$. Then the transformation group
$(\mathfrak{B},G)$ is not pseudo--co--decomposable to distal transformation semigroups.
\end{example}
\begin{proof}
Suppose $(\mathfrak{B},(S_\alpha:\alpha\in\Gamma))$ is a
pseudo--co--decomposition of $(\mathfrak{B},G)$ to distal
transformation semigroups, then for all $x\in\mathfrak{B}$ and
$\alpha\in\Gamma$, $xS_\alpha$ is finite (use Lemma~\ref{salam50}), thus for all $\alpha_1,\ldots,\alpha_n\in
\Gamma$ and $x\in \mathfrak{B}$, $xS_{\alpha_1}\cdots S_{\alpha_n}$ is finite.
\\
Consider $\sigma:\mathbb{N}\to\mathbb{N}$ with:
\[\cdots\mathop{\rightarrow}\limits^{\sigma}6
\mathop{\rightarrow}\limits^{\sigma}4
\mathop{\rightarrow}\limits^{\sigma} 2
\mathop{\rightarrow}\limits^{\sigma}1
\mathop{\rightarrow}\limits^{\sigma}3
\mathop{\rightarrow}\limits^{\sigma}5
\mathop{\rightarrow}\limits^{\sigma}7
\mathop{\rightarrow}\limits^{\sigma}\cdots\]
then $\mathsf{f}_\sigma\in G$ and there exist
$\beta_1,\ldots,\beta_m\in\Gamma$ with
$f_\sigma\in S_{\beta_1}\cdots S_{\beta_m}$, thus for all $k\geq1$
we have $1f_\sigma^k\in1\mathop{\underbrace{(S_{\beta_1}\cdots S_{\beta_m})\cdots(S_{\beta_1}\cdots S_{\beta_m})}}\limits_{k{\rm \: times}}=1S_{\beta_1}\cdots S_{\beta_m}$.
So
\[\{1,\frac13,\frac15,\ldots\}=\{1f_\sigma^k:k\geq1\}\subseteq
1S_{\beta_1}\cdots S_{\beta_m}\]
which is in contradiction with finiteness of
$1S_{\beta_1}\cdots S_{\beta_m}$ and
$(\mathfrak{B},G)$ is not pseudo--co--decomposable to
distal transformation semigroups.
\end{proof}
\begin{remark}\label{salam90}
There exist non--distal transformation groups co--decomposable to distal ones~\cite[Theorem~3.5]{decom}.
\end{remark}
\begin{corollary}\label{taha10}
Using Examples~\ref{salam70},~\ref{salam80} and
Remark~\ref{salam90}, we have the following strict inclusions:
(compare with~\cite[Theorem 3.5]{decom}): 
\vspace{5mm}
{\small\begin{center}
\begin{tabular}{|c|}
\hline The class  of all transformation semigroups\\
    \begin{tabular}{|c|} \hline
    The class  of all transformation semigroups pseudo--co--decomposable to distal ones\\
        \begin{tabular}{|c|} \hline
        The class  of all transformation semigroups co--decomposable to distal ones\\
            \begin{tabular}{|c|} \hline
            The class  of all distal transformation semigroups \\
            (Example: trivial transformation group $(X,\{id_X\})$) \\ \hline
            \end{tabular}\\
        Remark~\ref{salam90} \\ \hline
        \end{tabular}\\
    Example~\ref{salam70} \\ \hline
    \end{tabular}\\
Example~\ref{salam80} \\ \hline
\end{tabular}
\end{center}}
\end{corollary}
\begin{note}
One may study intraction of a transformation semigroup and its pseudo--co--decompositions via operators like
product, disjoint union, quotient of transformation semigroups using similar methods described in~\cite{decom}.
\end{note}
\subsection{A glance at minimality and point transitivity approach}
\noindent We say the transformation semigroup $(X,S)$ is point
transitive if there exists $x\in X$ such that $\overline{xS}=X$,
moreover $(X,S)$ is minimal if for each $x\in X$,
$\overline{xS}=X$. It's evident that $(X,S)$ is point transitive
(resp. minimal) if and only if $S$ has a sub--semigroup like
$S_0$ such that $(X,S_0)$ is point transitive (resp. minimal).
Hence $(X,S)$ is point transitive (resp. minimal) if and only if
it has a pseudo--co--decomposition
$(X,(S_\alpha:\alpha\in\Gamma))$ such that $(X,S_\alpha)$ is
point transitive (resp. minimal) for some $\alpha\in\Gamma$.
Hence what is interesting for us is pseudo--co--decomposability
of a point transitive (resp. minimal) transformation semigroup to
non--point transitive (resp. minimal) transformation semigroups.
\\
Consider unit circle $\mathbb{S}^1=\{e^{i\theta}:\theta\in\mathbb{R}\}$, also for $\alpha\in\mathbb{R}$ suppose:
\\
$\bullet$ $\mathop{\varphi_\alpha:\mathbb{S}^1\to\mathbb{S}^1}\limits_{\SP\SP\SP e^{i\theta}\mapsto e^{i(\alpha+\theta)}}$
    be $\alpha-$rotation in unit circle,
\\
$\bullet$ $\mathop{\varepsilon_\alpha:\mathbb{S}^1\to\mathbb{S}^1}\limits_{\SP\SP\SP e^{i\theta}\mapsto e^{i(\alpha-\theta)}}$
    be composition of $\alpha-rotation$ and
    conjugate map $\mathop{\eta:\mathbb{S}^1\to\mathbb{S}^1}\limits_{\SP e^{i\theta}\mapsto e^{-i\theta}}$.
\\
Then $\Sigma=\{\varphi_\alpha:\alpha\in2\pi\mathbb{Q}\}$ is the group of rational multiplication of $2\pi$
rotations on unit circle and $\Sigma^*=\{\varphi_\alpha:\alpha\in2\pi\mathbb{Q}\}\cup
\{\varepsilon_\alpha:\alpha\in2\pi\mathbb{Q}\}$ is the group generated by conjugate map and rational multiplication of $2\pi$
rotations on unit circle.
\begin{remark}\label{taha20}
$(\mathbb{S}^1,\Sigma)$ is a minimal (point transitive) transformation group, co--decomposable to non--minimal transformation
groups. In fact $(\mathbb{S}^1,(\{\varphi_\alpha^n:n\in\mathbb{N}\}:\alpha\in \{\frac{2\pi}{m}:m\geq1\}))$
is a co--decomposition of
$(\mathbb{S}^1,\Sigma)$ to non--minimal (non--point transitive) transformation groups~\cite[Counterexample 3.3]{decom}.
\end{remark}
\noindent $(\mathbb{S}^1,\Sigma^*)$ is minimal since $(\mathbb{S}^1,\Sigma)$ is minimal and $\Sigma\subseteq\Sigma^*$.
We show $(\mathbb{S}^1,\Sigma^*)$ is pseudo--co--decomposable to non--minimal transformation groups
and it is not co--decomposable to non--minimal transformation semigroups.
\\
Item (1) of the following lemma shows that $\Sigma^*$ is semigroup generated by $\eta$ and $\Sigma$. Since all elements
of $\Sigma^*$ have finite order,  $\Sigma^*$ is a group (under the operation of composition of maps) too. By item (3)
of the following Lemma $\Sigma^*$ is non--abelian.
\begin{lemma}\label{taha30}
We have:
\begin{itemize}
\item[1.] $\varepsilon_\alpha=\eta\varphi_\alpha=\varphi_{-\alpha}\eta$ for each $\alpha\in\mathbb{R}$,
\item[2.] $\{h\in\Sigma^*:h^2=id_{\mathbb{S}^1}\}=\{\varepsilon_\alpha:\alpha\in2\pi\mathbb{Q}\}\cup\{\varphi_\alpha:\alpha\in\pi\mathbb{Z}\}$,
\item[3.] $\{h\in\Sigma^*:h\eta=\eta h\}=\{\eta, id_{\mathbb{S}^1}, \varphi_\pi,\varepsilon_\pi\}$,
\item[4.] each sub--semigroup of $\Sigma^*$ is a sub--group of $\Sigma^*$.
\end{itemize}
\end{lemma}
\begin{proof}
(1) Consider $\theta,\alpha\in \mathbb{R}$, then
$e^{i\theta}\varepsilon_\alpha=e^{i(\alpha-\theta)}=e^{-i\theta}\varphi_\alpha=e^{i\theta}\eta\varphi_\alpha$
and
\linebreak
$e^{i\theta}\varepsilon_\alpha=e^{i(\alpha-\theta)}=e^{i(-\alpha+\theta)}\eta=e^{i\theta}\varphi_{-\alpha}\eta$.
\\
(2) For each $\alpha\in\mathbb{R}$, we have $\varphi_{2\alpha}=\varphi_\alpha^2=id_{\mathbb{S}^1}$ if and only if
$\alpha\in\pi\mathbb{Z}$, i.e. $\varphi_\alpha\in\{\varphi_\pi,\varphi_{2\pi}\}=\{id_{\mathbb{S}^1}, \varphi_\pi\}$ and
$\varepsilon_\alpha^2=\eta\varphi_\alpha\eta\varphi_\alpha=
\eta\varphi_\alpha\varphi_{-\alpha}\eta=\eta\eta\varphi_\alpha\varphi_{\alpha}^{-1}\eta=\eta^2=id_{\mathbb{S}^1}$.
\\
(3) Consider $\alpha\in\mathbb{R}$, then:
\begin{eqnarray*}
\varphi_\alpha\eta=\eta\varphi_\alpha & \Leftrightarrow & \eta\varphi_{-\alpha}=\eta\varphi_\alpha
        \Leftrightarrow \varphi_{-\alpha}=\varphi_\alpha \Leftrightarrow \varphi_{2\alpha}=id_{\mathbb{S}^1} \\
& \Leftrightarrow & \alpha\in\pi\mathbb{Z} \Leftrightarrow\varphi_\alpha\in\{id_{\mathbb{S}^1}, \varphi_\pi\}
\end{eqnarray*}
and
\begin{eqnarray*}
\varepsilon_\alpha\eta=\eta\varepsilon_\alpha & \Leftrightarrow & \eta\varphi_\alpha\eta=\eta\eta\varphi_\alpha
    \Leftrightarrow \varphi_\alpha\eta=\eta\varphi_\alpha \Leftrightarrow
    \varphi_\alpha\in\{id_{\mathbb{S}^1}, \varphi_\pi\} \\
& \Leftrightarrow & \varepsilon_\alpha=\eta\varphi_\alpha\in \{\eta id_{\mathbb{S}^1},\eta \varphi_\pi\}
    =\{\eta,\varepsilon_\pi\}
\end{eqnarray*}
(4) Use the fact that each element of $\Sigma^*$ has finite order.
\end{proof}
\begin{example}\label{taha40}
$(\mathbb{S}^1,(\{\eta^i\varphi_\alpha^n:n\in\mathbb{Z}, i=0,1\}:\alpha\in \{\frac{2\pi}{m}:m\geq2\}))$ is a
    pseudo--co--decomposition of $(\mathbb{S}^1,\Sigma^*)$ to
    non--minimal (non--point transitive) transformation groups.
\end{example}
\begin{proof}
For each $m\geq2$, $T_m=\{\eta^i\varphi_\frac{2\pi}{m}^n:n\in\mathbb{Z}, i=0,1\}=
\{\eta^i\varphi_\frac{2n\pi}{m}:1\leq n\leq m, i=0,1\}$ is a finite subgroup of $T$, hence $e^{i\theta}T_m$ is a finite
non--dense subset of $\mathbb{S}^1$ for each  $e^{i\theta}\in\mathbb{S}^1$. Therefore $(\mathbb{S}^1,T_m)$
is not point transitive. It is clear that $\Sigma^*=\bigcup\{T_m:m\geq2\}$. On the other hand for each
$s,t\geq1$ we have:
\begin{eqnarray*}
T_sT_t & = & \{\eta^i\varphi_\frac{2\pi n}{s}\eta^j\varphi_\frac{2\pi m}{t}:n,m\in\mathbb{Z}, 0\leq i,j\leq 1\} \\
& = & \{\eta^i\varphi_\frac{-2\pi m}{t}\eta^j\varphi_\frac{-2\pi n}{s}:n,m\in\mathbb{Z}, 0\leq i,j\leq 1\} \\
& = & \{\eta^i\varphi_\frac{2\pi m}{t}\eta^j\varphi_\frac{2\pi n}{s}:n,m\in\mathbb{Z}, 0\leq i,j\leq 1\}=T_tT_s
\end{eqnarray*}
\end{proof}
\begin{lemma}\label{taha50}
If $T$ is an infinite subgroup of $\Sigma^*$, then $(\mathbb{S}^1,T)$ is minimal.
\end{lemma}
\begin{proof}
Suppose  $T\subseteq\Sigma^*=\{\varphi_\alpha:\alpha\in2\pi(\mathbb{Q}\cap[0,1))\}\cup\{\varepsilon_\alpha:\alpha\in2\pi(\mathbb{Q}\cap[0,1))\}=\{\eta^i\varphi_\alpha:\alpha\in2\pi(\mathbb{Q}\cap[0,1)),i=0,1\}$ is an infinite
subgroup. There exists $i\in\{0,1\}$ and one--to--one sequence $\{q_n\}_{n\geq1}$ in $\mathbb{Q}\cap[0,1)$ such that
$\{\eta^i\varphi_{2\pi q_n}\}_{n\geq1}$ is a sequence in $T$. Choose $\kappa>0$, using compactness of $[0,1]$,
$\{q_n\}_{n\geq1}$ has a Cauchy subsequence, hence there exist $n,m\geq1$ such that $0<q_n-q_m<\frac{\kappa}{2\pi}$.
Moreover
\[T\ni\eta^i\varphi_{2\pi q_n}(\eta^i\varphi_{2\pi q_m})^{-1}=\eta^i\varphi_{2\pi (q_n-q_m)}\eta^i=\left\{\begin{array}{lc}
\varphi_{2\pi (q_n-q_m)} & i =0 \:, \\ \varphi_{2\pi (q_m-q_n)}=\varphi_{2\pi (q_n-q_m)}^{-1} & i=1\:.\end{array}\right.\]
In particular $\varphi_{2\pi (q_n-q_m)}\in T$. Consider $\theta,\mu\in\mathbb{R}$, there exists $j\in\mathbb{Z}$
such that $2\pi (q_n-q_m)j\leq\theta-\mu<2\pi (q_n-q_m)(j+1)$, i'e. $0\leq\theta-(\mu+2\pi (q_n-q_m)j)<2\pi (q_n-q_m)<\kappa$.
However $e^{i\lambda}:=e^{i(\mu+2\pi (q_n-q_m)j)}=e^{i\mu}\varphi_{2\pi (q_n-q_m)j}=e^{i\mu}\varphi_{2\pi (q_n-q_m)}^{j}\in e^{i\mu}T$, so:
\[\forall\kappa>0\exists\lambda\:\:(0\leq\theta-\lambda<\kappa\wedge e^{i\lambda}\in e^{i\mu}T)\]
which shows $e^{i\theta}\in\overline{e^{i\mu}T}$ for all $\theta\in\mathbb{R}$. Therefore
$\mathbb{S}^1=\overline{e^{i\mu}T}$ for all $\mu\in\mathbb{R}$ and $(\mathbb{S}^1,T)$ is minimal.
\end{proof}
\begin{theorem}\label{taha60}
$(\mathbb{S}^1,\Sigma^*)$ is not co--decomposable to
    non--minimal (non--point transitive) transformation semigroups.
\end{theorem}
\begin{proof}
Suppose $(\mathbb{S}^1,(H_\alpha:\alpha\in\Gamma))$ is a co--decomposition of $(\mathbb{S}^1,T)$ to non--point transitive
transformation semigroups. In particular for each $\alpha\in\Gamma$, $(\mathbb{S}^1,H_\alpha)$ is not minimal and
$H_\alpha$ is finite by Lemma~\ref{taha50}. There exist $\alpha_1,\ldots,\alpha_n\in\Gamma$ such that
$\eta\in S_{\alpha_1}\cdots S_{\alpha_n}$, thus for each $\alpha\in\Gamma\setminus\{\alpha_1,\ldots,\alpha_n\}$,
$s\in S_\alpha$ and $z\in \mathbb{S}^1$ we have $zs\eta=z\eta s$, i.e. $s\eta=\eta s$. By item (3) in Lemma~\ref{taha30},
$\bigcup\{S_\alpha:\alpha\in\Gamma\}\setminus\bigcup\{S_{\alpha_i}:1\leq i\leq n\}\subseteq \{\eta, id_{\mathbb{S}^1}, \varphi_\pi,\varepsilon_\pi\}$. Therefore $\bigcup\{S_\alpha:\alpha\in\Gamma\}\subseteq
S_{\alpha_1}\cup\cdots\cup S_{\alpha_n}\cup\{\eta, id_{\mathbb{S}^1}, \varphi_\pi,\varepsilon_\pi\}$.
Therefore $\bigcup\{S_\alpha:\alpha\in\Gamma\}$ is a finite subset of $\Sigma^*$. Using Lemma~\ref{taha30} and the fact that
each element of $\Sigma^*$ has finite order, sub--semigroup generated by $\bigcup\{S_\alpha:\alpha\in\Gamma\}$ is finite,
i.e. $\Sigma^*$ is finite which is a contradiction.
\end{proof}
\begin{example}\label{taha70}
By Example~\ref{taha40} and Theorem~\ref{taha60},
minimal transformation group $(\mathbb{S}^1,\Sigma^*)$ is pseudo--co--decomposable to non--point transitive
(non--minimal) transformation groups, however it is not co--decomposable to non--point transitive (non--minimal)
transformation semigroups.
\end{example}
\begin{corollary}
The following strict inclusions are valid: (compare with
~\cite[Theorem 3.6]{decom}): {\small\begin{center}
    \begin{tabular}{|c|} \hline
    The class  of all minimal transformation semigroups\\
        \begin{tabular}{|c|} \hline
        The class  of all minimal transformation semigroups \\ pseudo--co--decomposable to non--minimal ones\\
            \begin{tabular}{|c|} \hline
            The class of all minimal transformation semigroup \\ co--decomposable to non--minimal ones \\
            Remark~\ref{taha20} \\ \hline
            \end{tabular}\\
        Example~\ref{taha70} \\ \hline
        \end{tabular}\\
    (Example: trivial transformation group $(\{e\},\{id_{\{e\}}\})$) \\ \hline
    \end{tabular}
\end{center}}
\noindent The reader can replace ``(non--)minimal'' by ''(non--)point transitive'' in the above diagram.
\end{corollary}
\section{An arised question: strongly pseudo--co--decomposition of a transformation semigroup}
\noindent We say pseudo--co--decomposition
$(X,(S_\alpha:\alpha\in\Gamma))$ of $(X,S)$ is a strong
pseudo--co--decomposition of $(X,S)$ if for every
$\alpha_1,\ldots,\alpha_n \in\Gamma$ and permutation \linebreak
$\mathop{\{1,\ldots,n\}\to \{1,\ldots,n\}}\limits_{k\mapsto m_k}$
we have $S_{\alpha_1}S_{\alpha_2}\cdots S_{\alpha_n}=
S_{\alpha_{m_1}}S_{\alpha_{m_2}}\cdots S_{\alpha_{m_n}}$. One can
simply verify that in transformation semigroup $(X,S)$:
\begin{itemize}
\item if $S$ acts effectively on $X$, then any co--decomposition of $(X,S)$ is a strongly pseudo--co--decomposition of $(X,S)$,
\item every strongly pseudo--co--decomposition of $(X,S)$ is a pseudo--co--decomposition of $(X,S)$.
\end{itemize}
For dynamical property $\mathsf{P}$ we say the transformation
semigroup (group) $(X,S)$ is strongly pseudo--co--decomposable to $\mathsf{P}$
transformation semigroups (groups) if it has a strong pseudo--co--decomposition
like $(X,(S_\alpha:\alpha\in\Gamma))$ to transformation
semigroups (groups) such that for all $\alpha\in\Gamma$,
$(X,S_\alpha)$ has property $\mathsf{P}$, hence:
\begin{itemize}
\item by Example~\ref{salam10}, $(\mathfrak{B},T)$ is strongly pseudo--co--decomposable to
distal transformation groups,
\item in transformation semigroup $(X,S)$ with nontrivial $S$, let $M:=(S\times S)\sqcup\{\mathfrak{e}\}$ be a semigroup with identity $\mathfrak{e}$ and operation $(a,b)*(c,d)=(a,d)$ for each $(a,b),(c,d)\in S\times S$, then
$(X,M)$ is a (non--effective) transformation semigroup under action $x\mathfrak{e}:=x$ and $x(s,t):=xs$ (for $x\in X, s,t\in S$).
$(X,((S\times\{t\})\sqcup\{\mathfrak{e}\}:t\in S))$ is a pseudo--co--decomposition of $(X,M)$ and it is not a
strongly pseudo--co--decomposition of $(X,M)$.
\end{itemize}
In the class of all effective transformation (semi--)groups, $\mathcal C$, one may consider the following inclusions
(use Corollary~\ref{taha10} too):
{\small \begin{center}
\begin{tabular}{l}
The  class of all distal elements of $\mathcal C$  \\
$\subset$
The class of all elements of $\mathcal C$
co--decomposable to distal ones \\
$\subset$
The class of all elements of $\mathcal C$
strongly pseudo--co--decomposable to distal ones \\
$\mathop{\subseteq}\limits^{*}$
The class of all elements of $\mathcal C$
pseudo--co--decomposable to distal ones \\
$\subset$  $\mathcal C$.
\end{tabular}
\end{center}}
\noindent In the above chain of inclusions in order to have strict inclusions in
all cases we should have positive answer to the following question, which may be subject of next
research:
\begin{problem}
Is there any (effective) transformation (semi--)group $(X,S)$ pseudo--co--decomposable
to distal ones, which is not strongly pseudo--co--decomposable to distal ones?
\end{problem}

\noindent {\small
{\bf Safoura Arzanesh},
Faculty of Mathematics, Statistics and Computer Science,
College of Science, University of Tehran,
Enghelab Ave., Tehran, Iran
\\
({\it e-mail}: arzanesh.parsian@gmail.com)
\\
{\bf Fatemah Ayatollah Zadeh Shirazi},
Faculty of Mathematics, Statistics and Computer Science,
College of Science, University of Tehran,
Enghelab Ave., Tehran, Iran
\\
({\it e-mail}: f.a.z.shirazi@ut.ac.ir \& fatemah@khayam.ut.ac.ir)
\\
{\bf Reza Rezavand}, School of Mathematics, Statistics
and Computer Science, College of Science, University of Tehran,
Enghelab Ave., Tehran, Iran
\\
({\it e-mail}: rezavand@ut.ac.ir)

}

\end{document}